\documentclass[12pt,a4paper]{amsart}

\usepackage{amsmath,amsthm,amsfonts,amssymb,array,amscd}
\usepackage{amsmath,geometry,amssymb,amsfonts,amsthm,graphicx,enumerate,latexsym,tabularx,amscd}

\usepackage{refcheck} 
\norefnames
\nocitenames
\usepackage{geometry}
\usepackage{amsthm}
 \newtheorem{thm}{Theorem}[section]
 \newtheorem{prop}[thm]{Proposition}
 \newtheorem{lem}[thm]{Lemma}
 
 \newtheorem{cor}[thm]{Corollary}
\theoremstyle{definition}
 \newtheorem{exm}[thm]{Example}
 \newtheorem{dfn}[thm]{Definition}
 \theoremstyle{definition}
\theoremstyle{remark}
 \newtheorem{rem}[thm]{Remark}
 \numberwithin{equation}{section}
\theoremstyle{definition}
\theoremstyle{remark}
 \numberwithin{equation}{section}

\renewcommand{\le}{\leqslant}

\renewcommand{\setminus}{\smallsetminus}
\usepackage{tikz}
\usepackage[pdftex]{hyperref}
\usetikzlibrary{matrix,arrows}

\newcommand{\bbC}{\mathbb{C}}

\newcommand{\bbN}{\mathbb{N}}

\newcommand{\bbQ}{\mathbb{Q}}

\newcommand{\bbZ}{\mathbb{Z}}   

\renewcommand{\and}{\quad \mbox{and} \quad}  
\renewcommand{\le}{\leqslant}

\renewcommand{\setminus}{\smallsetminus}
\title{Invariant formula of the determinant of a Heisenberg representation}

\subjclass[2010]{20G05; 22E50\\Keywords: Transfer map, Heisenberg representations, Determinant}

\author[Biswas]{\bfseries Sazzad Ali Biswas}

\address{
Chennai Mathematical Institute\\ 
H1, Sipcot It Park, Siruseri  \\ 
Kelambakkam, 603103\\
India}
\email{sabiswas@cmi.ac.in, sazzad.jumath@gmail.com}

\thanks{The  author is partially supported by IMU-Berlin Einstein Foundation, Berlin, Germany and CSIR, Delhi, India }

\begin{document}

\vspace{18mm}
\setcounter{page}{1}
\thispagestyle{empty}

\begin{abstract}
In this paper we give an explicit formula of the determinant of a Heisenberg representation $\rho$ of a finite group
$G$. Heisenberg representations are induced by
1-dimensional characters in multiple ways, but our formula will be independent of any
particular choice of induction.
\end{abstract}

\maketitle

\section{\textbf{Introduction}}

Let $G$ be a finite group and $\rho$ be a Heisenberg representation of $G$, that is an irreducible
representation of the two-step nilpotent factor group $G/C^3(G)$. (cf. Subsection 2.1.) Let $Z$ be the
scalar group for $\rho$ and $H$ be any maximal isotropic subgroup of $G$ for $\rho$ (in the sense of
Remark \ref{Remark 2.2}(i) below). The central character $\chi_\rho:Z\to\bbC^\times$ of $\rho$ is then extendible to a
linear character $\chi_H$ of $H$ and for any choice of $(H,\chi_H)$ we have:
$$\rho\cong\rm{Ind}_{H}^{G}(\chi_H),\quad\det(\rho)=\det(\rm{Ind}_{H}^{G}(\chi_H)).$$
Our aim is to establish a formula of $\det(\rho)$ which is independent of the choices made.

We start with the following formula obtained by Gallagher's result (cf. Theorem \ref{Theorem Gall}):
$$\det(\rm{Ind}_{H}^{G}(\chi_H))(g)=\Delta_{H}^{G}(g)\cdot\chi_H(T_{G/H}(g))\quad\text{for all $g\in G$},$$
where $\Delta_{H}^{G}$ is the determinant of $\rm{Ind}_H^G(1_H)$. In fact $\Delta_{H}^{G}(s)$ is the 
sign $\varepsilon_{G/H}(s)$ of the
permutation of $G/H$ defined by left multiplication with $s$, and $T_{G/H}$ is the transfer map
from $G$ to $H$.

Since $\overline{G}=G/\rm{Ker}(\rho)$ is a two-step nilpotent group (cf. Remark \ref{Remark 2.2}(ii)) 
it will be enough
to exploit Gallagher's formula in the case when $G$ is two-step nilpotent. Thus in Lemma \ref{Lemma 3.1}
we compute $T_{G/H}$, when $H\supseteq [G,G]$ is an abelian subgroup of finite index of a two-step
nilpotent group $G$. Based on that, as our main result we will prove the following theorem
in section 4:

\begin{thm}\label{Theorem 1.1}
 Let $\rho=(Z,\chi_\rho)$ be a Heisenberg representation of $G$ of dimension $d$ and put
 \begin{itemize}
  \item $\varepsilon_\rho(g)=-1$ if $G/G^2Z$ is Klein's 4-group and $g\not\in G^2Z$,
  \item $\varepsilon_\rho(g)=1$ in all other cases (if $G/G^2Z$ is not Klein's 4-group or if $g\in G^2Z$).
 \end{itemize}
 Then we obtain
 \begin{equation}\label{eqn 1.1}
  \det(\rho)(g)=\varepsilon_\rho(g)\cdot\chi_\rho(g^d).
 \end{equation}
\end{thm}

We remark that all terms in equation (\ref{eqn 1.1}) depend only on $g\pmod{[G,G]}$, and hence can
be interpreted as functions on $G/[G,G]$. But only $\det(\rho)$ is always a character of $G/[G,G]$,
whereas the factors on the right need not be characters.

In section 5, we also give an arithmetic version of the above theorem (see Proposition \ref{Proposition 5.1}).



\section{\textbf{Notations and preliminaries}}

\subsection{Heisenberg Representations}

Let $\rho$ be an irreducible representation of a (pro-)finite group $G$. Then $\rho$ is called a \textbf{Heisenberg 
representation} if it represents commutators by 
scalar operators. Therefore higher commutators are represented by $1$.
Let $Z_\rho$ be the \textbf{scalar} group of $\rho$, i.e., $Z_\rho\subseteq G$ and $\rho(z)=\text{scalar matrix}$
for every $z\in Z_\rho$. We can see that the linear characters of $G$ are Heisenberg representations as the degenerate special case.
If $C^1G=G$, $C^{i+1}G=[C^iG,G]$ denotes the 
descending central series of $G$, the Heisenberg
property means $C^3G\subset \mathrm{Ker}(\rho)$, 
 and therefore $\rho$ determines a character $X$ on the alternating square of $A:=G/C^2G$ such that:
\begin{equation}
\rho([\hat{a_1},\hat{a_2}])= X(a_1,a_2)\cdot E
\end{equation}
for $a_1,a_2 \in A $ with lifts $\hat{a_1},\hat{a_2}\in G$. The equivalence class of $\rho$ is determined by the
projective kernel 
$Z_{\rho}$ which has the property that $Z_{\rho}/C^2G$ is the \textbf{radical}
of $X$ and by the character $\chi_{\rho}$ of $Z_{\rho}$ such 
that $\rho(g)=\chi_{\rho}(g)\cdot E$ for all $g\in Z_{\rho}$ and $E$ being the unit operator.
 Here $\chi_{\rho}$ is a $G$-\textbf{invariant} character of $Z_{\rho}$ 
 which we call the central character of $\rho$.

It is known that the Heisenberg representations $\rho$ are fully characterized by 
the corresponding pairs $(Z_{\rho},\chi_{\rho})$.
\begin{prop}[\textbf{\cite{Z3}, Proposition 4.2}]\label{Proposition 3.1}
The map $\rho\mapsto(Z_\rho,\chi_\rho)$ is a bijection between equivalence 
classes of Heisenberg representations of $G$ and the pairs $(Z_\rho,\chi_\rho)$ such that 
\begin{enumerate}
 \item[(a)] $Z_\rho\subseteq G$ is a coabelian normal subgroup,
 \item[(b)] $\chi_\rho$ is a $G$-invariant character of $Z_\rho$,
 \item[(c)] $X(\hat{g_1},\hat{g_2}):=\chi_\rho(g_1g_2g_1^{-1}g_2^{-1})$ is a nondegenerate 
 \textbf{alternating character} on $G/Z_\rho$, where $\hat{g_1},\hat{g_2}\in G/Z_{\rho}$ and their 
 corresponding lifts $g_1,g_2\in G$.
\end{enumerate}
\end{prop}
\begin{rem}\label{Remark 2.2}
{\bf (i).}
For pairs $(Z_\rho,\chi_\rho)$ with the properties $(a)-(c)$, the corresponding Heisenberg representation $\rho$ is 
determined 
by the identity:
\begin{equation}\label{eqn 3.2}
 \sqrt{[G:Z_\rho]}\cdot\rho=\mathrm{Ind}_{Z_\rho}^{G}\chi_\rho.
\end{equation}
Furthermore, if $H$ is a maximal isotropic subgroup of $G$ for $\rho$ ({\bf which is to say that
$H/Z_\rho\subset G/Z_\rho$ is maximal isotropic with respect to $X$}) and if $\chi_H\in\widehat{H}$ is an extension
of $\chi_\rho$, then we have (cf. \cite{Z3}, p.193, Proposition 5.3, more properly: \cite{EWZ}, p.271/2, Lemma 3
and Proposition 3):
\begin{center}
 $\rho=\rm{Ind}_{H}^{G}\chi_H$ and $[G:H]=[H:Z_\rho]=\sqrt{[G:Z_\rho]}=\rm{dim}(\rho)$.
\end{center}
{\bf (ii).}
Property (a) means: $Z_\rho \supseteq [G,G],$ which implies $[G,Z_\rho]\supseteq [G,[G,G]].$\\
Property (b) means: $[G,Z_\rho] \subseteq \rm{Ker}(\chi_\rho)$, which implies 
$[G,Z_\rho]\subseteq [G,G]\cap \rm{Ker}(\chi_\rho).$\\
Because $\rho(g)=\chi_\rho(g)\cdot E$ for $g\in Z_\rho$ and because of (c) we obtain that
$$ \chi_\rho (g_1g_2g_1^{-1}g_2^{-1})\cdot E = [\rho(g_1),\rho(g_2)]$$
is nondegenerate on $G/Z_\rho$ hence $\rm{Ker}(\rho)\subseteq Z_\rho$ and therefore 
$\rm{Ker}(\rho)=\rm{Ker}(\chi_\rho)\supseteq [G,Z_\rho].$\\
{\bf (iii).} Actually we see now that $\rho$ is a representation of the two-step nilpotent 
group $G/[G,Z_\rho]$ which has  
$Z_\rho/[G,Z_\rho]$ as its center. \\
{\bf (iv).} If we go further to $\overline{G}:=G/\rm{Ker}(\rho)$ then in $\overline{G}$ the relation 
$\chi_\rho(g_1g_2g_1^{-1}g_2^{-1})=1$ implies already $g_1g_2g_1^{-1}g_2^{-1}=1$ because 
$\rm{Ker}(\rho)=\rm{Ker}(\chi_\rho)$. Therefore in $\overline{G}$ isotropic is the same as commuting elements,
which implies that $\chi_\rho$ extends to $\chi_H$ if $H$ is maximal isotropic.
\end{rem}

\subsection{Transfer map}

Let $H$ be a subgroup of a finite group $G$. Let $\{t_1,t_2,\cdots,t_n\}$ be a left transversal for $H$ in $G$. 
If $g\in G$ then 
for all $i=1,2,\cdots,n$ we obtain, 
\begin{equation}
 g t_i\in t_{g(i)} H,
\end{equation}
where the map $i\mapsto g(i)$ is a well-defined permutation of the set $\{1,2,\cdots,n\}$. Assume that $f:H\to A$ is
a homomorphism from $H$ to an abelian
group $A$. Then \textbf{transfer} of $f$, written $T_{f}$, is a mapping 
\begin{center}
 $T_f:G\to A$ \hspace{.5cm}given by \\
 $T_{f}(g)=\prod_{i=1}^{n}f(t_{g(i)}^{-1}g t_{i})$ \hspace{.4cm}for all $g\in G$.
\end{center}
Since $A$ is abelian, the order of the factors in the product is irrelevant. Now we take $f$ the canonical homomorphism, i.e.,
\begin{center}
 $f:H\to H/[H,H]$, where $[H,H]$ is the commutator subgroup of $H$.
\end{center}
And we denote $T_{f}=T_{G/H}$. Thus by definition of transfer map $T_{G/H}:G\to H/[H,H]$, we have  
\begin{equation}
 T_{G/H}(g)=\prod_{i=1}^{n}f(t_{g(i)}^{-1}g t_{i})=\prod_{i=1}^{n}t_{g(i)}^{-1}g t_{i}[H,H],
\end{equation}
for all $g\in G$.

Moreover, if $H$ is any subgroup
of finite index in $G$, then (cf. \cite{AT}, Chapter 13, p. 183)  
\begin{equation}\label{eqn 1.3}
 T_{G/gHg^{-1}}(g')=g T_{G/H}(g') g^{-1}, 
\end{equation}
for all $g,g'\in G$.
Now let $H$ be an abelian normal subgroup of $G$. Let $H^{G/H}$ be the set 
consisting the elements which are invariant under conjugation. So it is clear that these elements
are central elements and $H^{G/H}\subseteq Z(G)$, the center of $G$. When $H$ is abelian normal subgroup of $G$,
from equation (\ref{eqn 1.3}) we can conclude that (cf. \cite{AT}, Chapter 13, p. 183) that
\begin{equation}\label{eqn 2.6}
 \mathrm{Im}(T_{G/H})\subseteq H^{G/H}\subseteq Z(G). 
\end{equation}



 

 To compute the determinant of an induced representation of a finite group, we need the following theorem.
 \begin{thm}[Gallagher, \cite{GK}, Theorem $30.1.6$]\label{Theorem Gall}
 Let $G$ be a finite group and $H$ a subgroup of $G$. Let $\rho$ be a representation of $H$ and denote 
 $\Delta_{H}^{G}=\mathrm{det}(\mathrm{Ind}_{H}^{G}1_H)$. Then 
  \begin{equation}
   \mathrm{det}(\mathrm{Ind}_{H}^{G}\rho)(g)=(\Delta_{H}^{G})^{\mathrm{dim}(\rho)}(g)\cdot
   (\mathrm{det}(\rho)\circ T_{G/H})(g), \quad\text{for all $g\in G$}.
  \end{equation}


\end{thm}

Let $T$ be a left transversal for $H$ in $G$. Here $\mathrm{Ind}_{H}^{G}\rho$
is a block monomial representation (cf. \cite{GK}, p. 956) with block positions indexed by pairs $(t,s)\in T\times T$.
For $g\in G$, the $(t,s)$-block of $\mathrm{Ind}_{H}^{G}\rho$ is zero unless $gt\in sH$, i.e., $s^{-1}gt\in H$ and in 
which case
the block equal to $\rho(s^{-1} g t)$. 
Then we can write for $g\in G$
\begin{equation}
 T_{G/H}(g)=\prod_{t\in T}s^{-1}gt[H,H].
\end{equation}
Thus for all $g\in G$
\begin{align}
 \mathrm{det}(\mathrm{Ind}_{H}^{G}\rho)(g)\nonumber
&=(\Delta_{H}^{G})^{\mathrm{dim}(\rho)}(g)\cdot\mathrm{det}(\rho)\circ T_{G/H}(g)\\\nonumber
&=(\Delta_{H}^{G})^{\mathrm{dim}(\rho)}(g)\cdot\mathrm{det}(\rho)(\prod_{t\in T}s^{-1}gt[H,H])\\
&=(\Delta_{H}^{G})^{\mathrm{dim}(\rho)}(g)\cdot\prod_{t\in T}\mathrm{det}(\rho)(s^{-1}gt [H,H]),\label{eqn 1.7}
\end{align}
where in each factor on the right, $s=s(t)$ is uniquely determined by $gt\in sH$.


%


\subsection{Some useful results from finite Group Theory}

Let $G$ be a finite abelian group and put $\alpha=\prod_{g\in G}g$. By the following theorem we can compute $\alpha$.

\begin{thm}[\cite{PC}, Theorem 6 (Miller)]\label{Theorem 2.4}
 Let $G$ be a finite abelian group and $\alpha=\prod_{g\in G}g$.
 \begin{enumerate}
  \item If $G$ has no element of order $2$, then $\alpha=e$.
  \item If $G$ has a unique element $t$ of order $2$, then $\alpha=t$.
  \item If $G$ has at least two elements of order $2$, then $\alpha=e$.
 \end{enumerate}

\end{thm}

Let $G$ be a two-step\footnote{Its derived subgroup, i.e., commutator subgroup
$[G,G]$ is contained in its center. 
In other worlds, $[G,[G,G]]=\{1\}$, i.e., any triple commutator gives identity.} nilpotent group. 
For $G$, we have the following 
lemma which follows from the properties of two-step nilpotent groups and Lemma 9 on p. 77 of \cite{AB}. 
\begin{lem}\label{Lemma 2.4}
 Let $G$ be a finite two-step-nilpotent group and $x,y\in G.$ Then:\\
{\bf (i)} $[x,y]:=xyx^{-1}y^{-1}$ is always in the center of $G.$\\
{\bf (ii)} $[x_1x_2,y]=[x_1,y][x_2,y].$\\
{\bf (iii)} $[x^n,y]= [x,y]^n$ for all $n\in \bbN.$\\
{\bf (iv)} $x^ny^n =(xy)^n\cdot [x,y]^{n(n-1)/2}$ for all $n\in \bbN,$\\
{\bf (v)} $x_1^n\cdots x_s^n = (x_1\cdots x_s)^n\cdot\prod_{1\le i< j\le s}
[x_i,x_j]^{n(n-1)/2}$ for all $n\in \bbN.$
\end{lem}


\begin{dfn}[\textbf{$2$-rank of a finite abelian group}]
Let $G$ be a finite abelian group. Then from elementary divisor theorem (cf. \cite{DF} p. 77, Theorem 3) we can write
\begin{equation}
 G\cong\bbZ_{m_1}\times\bbZ_{m_2}\times\cdots\times\bbZ_{m_s}
\end{equation}
where $m_1|m_2|\cdots|m_s$ and $\prod_{i=1}^{s}m_i=|G|$. We define the $2$-rank of $G$
\begin{center}
$\rm{rk}_2(G):=$the number of $m_i$-s which are even.
\end{center}
In other words,
a finite abelian group $G$ is the direct product of its p-Sylow groups $G(p)$, and  the 2-rank of $G$ is equal to 
the $\mathbb{F}_2$-dimension of the subgroup $G[2]\subseteq G(2)$ of elements of orders $\le 2$.
\end{dfn}

\begin{prop}\label{Proposition 212}
 Let $G$ be an abelian group of $\mathrm{rk}_2(G)=n$. Then $G$ has $2^n-1$ nontrivial elements of order $2$. 
\end{prop}
\begin{proof}
 By the given condition, $\rm{rk}_2(G)=n$, i.e., the vector space $G[2]$ over $\mathbb{F}_2$ has $2^n$ elements
 and they are of order $2$.
 %
Hence we can conclude that when $G$ is abelian with $\mathrm{rk}_{2}(G)=n$, 
it has $2^n-1$ nontrivial elements of order $2$.
 
\end{proof}

We also need a structure theorem for finite abelian groups which come provided with an alternating character:

\begin{lem}[\cite{EWZ}, p. 270, Lemma 1(VI)]\label{Lemma 2.8}
 Let $G$ be a finite abelian group and assume the existence of an alternating bi-character 
 $X:G\times G\to\bbC^\times$  ( $X(g,g)=1$ for all $g\in G$, hence $1=X(g_1g_2,g_1g_2)=X(g_1,g_2)\cdot X(g_2,g_1)$) 
 which is nondegenerate. Then there will exist elements $t_1, t_1',\cdots,t_s,t_s'\in G$ 
such that 
 \begin{enumerate}
  \item 
   $G=<t_1>\times<t_1'>\times\cdots\times<t_s>\times<t_s'>$\\
  $ \cong\bbZ/m_1\times\bbZ/m_1\times\cdots\times\bbZ/m_s\times\bbZ/m_s$
  and $m_1|\cdots|m_s$.
 
\item For all $i=1,2,\cdots,s$ we have $X(t_i,t_i')=\zeta_{m_i}$ a primitive $m_i$-th root of unity.
  \item If we say $g_1\perp g_2$ if $X(g_1,g_2)=1$, then $(<t_i>\times<t_i'>)^\perp=\prod_{j\ne i}(<t_j>\times<t_j'>)$.
 \end{enumerate}

\end{lem}

\section{{\bf Computation of the transfer map for two-step nilpotent group}}

For two-step nilpotent group, we have the following lemma.
\begin{lem}\label{Lemma 3.1}
Assume that $G$ is a two-step nilpotent group and $H\supseteq [G,G]$ is an abelian subgroup of index 
$[G:H]=d.$ Then the transfer map $T_{G/H}:G/[G,G] \rightarrow H^{G/H}$ restricted to $H$ is given as
$$  T_{G/H}(h) = h^d\cdot [h,\alpha_{G/H}],\quad\text{for all}\quad h\in H,$$
where $\alpha_{G/H} = \prod_{\overline{g}\in G/H} \overline{g}  \in G/H,$ hence $[h,\alpha_{G/H}]$ is a central 
element of order $\le 2.$\\
As a consequence also $h^d$ is always a central element: $h^d \in Z(G).$
\end{lem}

\begin{proof} 
Let $T$ be a left transversal for $H$. When $h\in H$,  
 we have 
 \begin{center}
  $ht=t\cdot t^{-1}ht\in tH$,
 \end{center}
because $H\supseteq [G,G]$, hence $H$ is a normal subgroup of 
 $G$. Hence $s=t$, where $s=s(t)$ is a function of $t$ which is uniquely determined by $gt\in sH$, for some 
 $g\in G$. Therefore:
 \begin{align}
  T_{G/H}(h)\nonumber
  &=\prod_{t\in T}s^{-1}ht=\prod_{t\in T}t^{-1}ht=\prod_{t\in T}hh^{-1}t^{-1}ht=\prod_{t\in T}(h\cdot[h^{-1},t^{-1}])\\
  &=h^{d}\prod_{t\in T}[h^{-1},t^{-1}]=h^{d}[h^{-1},\prod_{t\in T}t^{-1}]=h^d\cdot[h,\alpha_{G/H}],\label{eqn 3.1}
 \end{align}
 where $\alpha_{G/H}=\prod_{t^{-1}\in T}t=\prod_{\bar{g}\in G/H}\bar{g}$.

Moreover the commutator $[h,\alpha_{G/H}]$ is well defined because $H$ is abelian and $[.,.]$ is bilinear.
Then $[h,\alpha_{G/H}]$ is of order $\le 2$ because, by the Miller's result (cf. Theorem \ref{Theorem 2.4}), 
$\alpha_{G/H}\in G/H$ is of order $\le 2$.
Finally $h^d\in Z(G)$ follows from $T_{G/H}(h),\;[h,\alpha_{G/H}]\;\in Z(G).$\\

\end{proof}

\section{\textbf{Proof of Theorem \ref{Theorem 1.1}}}


To prove Theorem \ref{Theorem 1.1}, we need to check the following properties of $\varepsilon_\rho$:
\begin{enumerate}
 \item $\varepsilon_\rho$ has values in $\{\pm 1\}$.
 \item $\varepsilon_\rho(gx)=\varepsilon_\rho(g)$ for all $x\in G^2Z$, hence $\varepsilon_\rho$ is a
 function on the factor group
 $G/G^2Z$, and in particular, $\varepsilon_\rho\equiv 1$ if $[G:Z]=d^2$ is odd.
 \item  The sign function $\varepsilon_\rho(g)$ is $\equiv 1$ unless $G/G^2Z$ is Klein's 4-group in which case we 
 have $\varepsilon_\rho(g)=-1$ if and only if $g\notin G^2Z.$


\end{enumerate}
We also need the following lemma.

\begin{lem}\label{Lemma 4.1}
 Let $\rho=(Z,\chi_\rho)$ be a Heisenberg representation of $G$ and put $X_\rho(g_1,g_2):=\chi_\rho([g_1,g_2])$.
 Then for every element $g\in G$, there exists a maximal isotropic subgroup $H$ for $X_\rho$ such that 
 $g\in H$.
\end{lem}
\begin{proof}

Let $g$ be a nontrivial element in $G$. Now we take a cyclic subgroup $H_0$ generated by $g$, i.e., $H_0=<g>$. Then 
$X_\rho(g,g)=1$ implies $H_0\subseteq H_{0}^{\perp}$. If $H_0$ is not maximal isotropic, then the inclusion is proper
and $H_0$ together
with some $h\in H_{0}^{\perp}\setminus H_0$ generates some larger isotropic subgroup $H_1\supset H_0$. Again we have 
$H_1\subseteq H_{1}^{\perp}$, and if $H_1$ is not maximal then the inclusion is proper, then again we proceed same
method and will 
have another isotropic subgroup and we continue this process step by step  come to maximal isotropic subgroup $H$.

Therefore for every element $g\in G$, we would have a maximal subgroup $H$ such that $g\in H$.    
\end{proof}

Now let $\rho= (Z,\chi_\rho)$ be a Heisenberg representation of $G$ which is of dimension $d=\sqrt{[G:Z]}.$ Then
 Theorem \ref{Theorem Gall} yields
$$  \det(\rho)(g) = \Delta_H^G(g)\cdot \chi_H(T_{G/H}(g))= \Delta_H^G(g)\cdot \chi_\rho(T_{G/H}(g)),$$
for any $H\supset Z$ such that $H/Z$ is maximal isotropic for $(G/Z, X_\rho)$ and for any character $\chi_H$ 
such that $\chi_H|_Z =\chi_\rho.$ We may assume here that $H\supset Z$ are abelian groups
(see Remark \ref{Remark 2.2}(iv)) and then $T_{G/H}(g)\in H^{G/H}\subseteq Z,$ hence the second equality follows. 
In particular
\begin{equation}\label{eqn 4.1}
  \det(\rho)(g) = \chi_\rho(T_{G/H}(g)),
\end{equation}
if $g\in H$ because $\Delta_H^G = \Delta_{\{1\}}^{G/H}.$\\


\begin{lem}\label{Lemma 4.2}
 Let $\rho=(Z,\chi_\rho)$ be a Heisenberg representation of dimension $d$ of $G$ and 
let $\overline{G}:=G/Ker(\rho).$ Then $\overline{G}$ is two-step nilpotent with center 
$Z(\overline{G})=\overline{Z}\supset [\overline{G},\overline{G}].$ 
And the maximal abelian subgroups $\overline{H}\subset \overline{G}$ correspond to the maximal isotropic subgroups $H/Z$
in $(G/Z, X_\rho).$ Moreover we have:\\
{\bf (i)} Each $g\in \overline{G}$ is contained in some maximal abelian
$\overline{H}\subset \overline{G},$ and $[\overline{G}:\overline{H}]= [G:H]= d.$\\
{\bf (ii)} If $g\in\overline{H}$ as in (i), then
$$ det(\rho)(g) = \chi_\rho(g^d)\cdot \chi_\rho([g,\alpha_{G/H}]) = \chi_\rho(g^d)\cdot X_\rho(g,\alpha_{G/H}).$$
In particular $det(\rho)(g)=\chi_\rho(g^d)$ if $g\in\overline{H}$ such that $\rm{rk}_2(G/H)\ne 1.$\\
{\bf (iii)} We have
$$  det(\rho)(g) =\pm \chi_\rho(g^d) = \varepsilon_\rho(g)\cdot \chi_\rho(g^d)$$
for all $g\in\overline{G},$ and $det(\rho)(g) = \chi_\rho(g^d)$ if $d$ is odd.\\
({\bf Note:} Part (iii) is nothing else than formula (\ref{eqn 1.1}) and part (1) of Theorem \ref{Theorem 1.1}.)

\end{lem}

\begin{proof}


From Remark \ref{Remark 2.2}(ii) (iii), (iv);  the preliminary part of this lemma is clear.\\
{\bf (i).} 
 From Lemma \ref{Lemma 4.1} we see that any $g\in \overline{G}$ must sit in a maximal isotropic 
subgroup $\overline{H}$
and by the preliminary remarks this is equivalent to $\overline{H}$ being abelian.\\
And since $\overline{H}$
corresponds a maximal isotropic subgroup $H$ for $X_\rho$, it can be seen that 
$[\overline{G}:\overline{H}]=[G:H]=d$.

{\bf (ii)/(iii)} If $g\in\overline{H}$ as in (i) then we have seen already that 
$\det(\rho)(g) = \chi_\rho(T_{G/H}(g)).$
And by equation (\ref{eqn 3.1}) we have $T_{G/H}(g)=g^d\cdot [g,\alpha_{G/H}].$ Together this yields
$$  \det(\rho)(g) =\chi_\rho([g,\alpha_{G/H}])\cdot \chi_\rho(g^d)  = \varepsilon_\rho(g)\chi_\rho(g^d)$$
because $[g,\alpha_{G/H}]$ is of order $\le 2,$ hence $\chi_\rho([g,\alpha_{G/H}])$ must be a sign.
And by Theorem \ref{Theorem 2.4}. we know that $\alpha_{G/H}= \overline{1} \in G/H$ if and only if $\rm{rk}_2(G/H)\ne 1.$

\end{proof}

\begin{lem}
The sign function $\varepsilon_\rho:\overline{G} \rightarrow \{\pm 1\}$ is actually a function
on $\overline{G}/\overline{G}^2 \overline{Z} = G/G^2 Z.$\\
({\bf Note:} This lemma is the same as part (2) of Theorem \ref{Theorem 1.1}.)
\end{lem}

\begin{proof}[{\bf Proof of Theorem \ref{Theorem 1.1}(2)}]
Because $Z$ is the scalar group of the irreducible representation $\rho$ of dimension $d$, then by
definition of scalar 
group, elements $z\in Z$ are 
represented by scalar matrices, i.e., 
\begin{equation*}
 \rho(z)=\chi_\rho(z)\cdot I_d, \quad \text{where $I_d$ is the $d\times d$ identity matrix}.
\end{equation*}
This implies 
\begin{center}
 $(\det\rho)(z)=\chi_\rho(z)^{d}=\chi_\rho(z^d)$.
\end{center}
We also know that $Z$ is the radical of $X_\rho$, therefore 
\begin{center}
 $X_\rho(z,g)=\chi_\rho([z,g])=1$ for all $z\in Z$ and $g\in G$.
\end{center}
Moreover, we can consider $\det\rho$ as a linear character of $G$, therefore 
\begin{equation}\label{eqn 2.17}
 (\det\rho)(gz)=(\det\rho)(g)\cdot(\det\rho)(z)=\varepsilon_\rho(g)\chi_\rho(g^d)\chi_\rho(z^d).
\end{equation}
On the other hand 
\begin{equation}\label{eqn 2.18}
 (\det \rho)(gz)=\varepsilon_\rho(gz)\chi_\rho((gz)^d)
 =\varepsilon_\rho(gz)\chi_\rho(g^dz^d[g,z]^{-\frac{d(d-1)}{2}})=\varepsilon_\rho(gz)\chi_\rho(g^d)
 \chi_\rho(z^d).
\end{equation}
On comparing equations (\ref{eqn 2.17}) and (\ref{eqn 2.18}) we get 
\begin{center}
 $\varepsilon_\rho(gz)=\varepsilon_\rho(g)$ for all $g\in G$ and $z\in Z$.
\end{center}
Moreover, since $\varepsilon_\rho(g)$ is a sign, we have
\begin{equation*}
 (\det\rho)(g^2)=(\det\rho)(g)^2=\varepsilon_\rho(g)^2\chi_\rho(g^d)^2=\chi_\rho(g^{2d}).
\end{equation*}
Therefore
\begin{equation}\label{eqn 2.19}
 (\det\rho)(gx^2)=(\det\rho)(g)\cdot(\det\rho)(x^2)=\varepsilon_\rho(g)\chi_\rho(g^d)\chi_\rho(x^{2d}).
\end{equation}
On the other hand
\begin{equation}\label{eqn 2.201}
 (\det\rho)(gx^2)=\varepsilon_\rho(gx^2)\chi_\rho((gx^2)^d)=\varepsilon_\rho(gx^2)\chi_\rho(g^d)\chi_\rho(x^{2d}),
\end{equation}
because $[g,x^2]^{\frac{d(d-1)}{2}} = [g,x^d]^{d-1}$ and $[g, x^d]\in [G,Z_\rho]\subseteq Ker(\rho).$\\
So we see from equations (\ref{eqn 2.19}) and (\ref{eqn 2.201}) $\varepsilon_\rho(gx^2)=\varepsilon_\rho(g)$, hence 
$\varepsilon_\rho$ is a function on $G/G^2Z$.
 
In particular, when $[G:Z]$ is odd, we have $\varepsilon_\rho\equiv 1$ because $[G:Z]$ odd implies
$G=G^2 Z$.\\

\end{proof}

\begin{proof}[{\bf Proof of Theorem \ref{Theorem 1.1}}]{\bf (3)}
We show that $\varepsilon_\rho(g)\equiv 1$ unless $G/G^2 Z$ is Klein's 4-group. We may assume 
that $d$ is even because otherwise $G=G^2 Z.$

Since $G/Z$ is abelian and is provided with nondegenerate alternating character $X_\rho$ we may apply 
Lemma \ref{Lemma 2.8}
to the pair $(G/Z, X_\rho).$\\
Then we see that
$$  G/Z = H/Z \times H'/Z$$
is the product of two isomorphic maximal isotropic subspaces, hence $\rm{rk}_2(G/Z)$ is always even and
$\rm{rk}_2(G/Z)\ne 2$
is the same as
$\rm{rk}_2(G/H') = \rm{rk}_2(G/H) \ne 1.$ Then from Proposition \ref{Proposition 212} we can 
say both $G/H$ and $G/H'$ have at least $3$ elements of order $2$. Then from Theorem \ref{Theorem 2.4} we have 
 $\alpha_{G/H}=1$ and  $\alpha_{G/H'}=1$.
Furthermore, from formula (\ref{eqn 3.1}) we obtain 
\begin{equation*}
 T_{G/H}(h)=h^d\cdot[h,\alpha_{G/H}]=h^d,\quad \text{and}\quad T_{G/H'}(h')=h'^d\cdot[h',\alpha_{G/H'}]=h'^d.
\end{equation*}
Therefore by using equation (\ref{eqn 4.1}) we have 
$$\det(\rho)(h)=\chi_\rho(h^d),\quad\text{and}\quad\det(\rho)(h')=\chi_\rho(h'^d).$$
So we can write 
\begin{align*}
 (\det\rho)(g)
 &=(\det\rho)(h)\cdot(\det\rho)(h'),\quad\text{here $g=h\cdot h'$ is a decomposition of $g$ with $h\in H,$ $h'\in H'$},\\
 &=\chi_\rho(h^d)\cdot\chi_\rho(h'^d)\\ 
 &=\chi_\rho(h^d\cdot h'^d)\\
 &=\chi_\rho((h\cdot h')^d[h,h']^{\frac{d(d-1)}{2}})\quad\text{using Lemma $\ref{Lemma 2.4}(iv)$}\\
 &=\chi_\rho(g^d)\cdot X_\rho(h,h')^{\frac{d(d-1)}{2}}\\
 &=\chi_\rho(g^d),
\end{align*}
because our assumptions $G\ne G^2Z,$ $\rm{rk}_2(G/Z)\ne 2$ and Lemma \ref{Lemma 2.8} applied to $(G/Z, X_\rho)$ imply
$m_1\cdots m_s =d$ with at least two even factors $m_\nu$, hence $m_i|\frac{d}{2}$ for all $i\in\{1,2,...,s\},$ 
and then 
\begin{center}
 $X_\rho(h,h')^{\frac{d(d-1)}{2}}=\chi_\rho([h,h'])^{\frac{d(d-1)}{2}}=\zeta_m^{\frac{d(d-1)}{2}}=1$,
\end{center}
where $\zeta_m$ is a primitive $m$-th root of unity and $m$ is some positive integer (which is the order of $[h,h']$) 
 which divides $\frac{d}{2}$.
This shows that when $\mathrm{rk}_2(G/Z)\ne 2$ we 
have $\varepsilon_\rho\equiv 1$.

{\bf Finally we are left} with the case $\rm{rk}_2(G/Z)=2.$ The Lemma \ref{Lemma 2.8} for $(G/Z,X_\rho)$ means now 
$m_1\cdots m_s =d$ where only $m_s$ is even and all other $m_\nu$ are odd. With the notation of 
Lemma \ref{Lemma 2.8} put
$$ H/Z =\langle t_1\rangle \times\cdots \langle t_s\rangle,\qquad 
H'/Z =\langle t_1'\rangle \times\cdots \langle t_s'\rangle.$$
Then $\rm{rk}_2(G/H)=\rm{rk}_2(G/H')=1$ and
$$ \varepsilon_\rho(\widehat{t_s}) = \chi_\rho([t_s,\alpha_{G/H}])=-1,\qquad \varepsilon_\rho(\widehat{t_s'}) = 
\chi_\rho([t_s',\alpha_{G/H'}])=-1$$

From Lemma \ref{Lemma 4.2}(iii) we see:
$$  1= \frac{(det\;\rho)(g_1)\cdot (det\;\rho)(g_2)}{(det\;\rho)(g_1 g_2)} =
\frac{\varepsilon_\rho(g_1)\chi_\rho(g_1^d)\cdot \varepsilon_\rho(g_2)\chi_\rho(g_2^d)}{\varepsilon_\rho(g_1 g_2)
\chi_\rho((g_1 g_2)^d)}$$
hence
$$ \frac{\varepsilon_\rho(g_1)\varepsilon_\rho(g_2)}{\varepsilon_\rho(g_1g_2)} = 
\frac{\chi_\rho((g_1g_2)^d)}{\chi_\rho(g_1^d g_2^d)} =
X_\rho(g_1,g_2)^{\frac{d(d-1)}{2}} = X_\rho(g_1,g_2)^{\frac{d}{2}},$$
because $d-1$ is odd and the value of $X_\rho^{\frac{d}{2}}$ must be a sign. Thus we conclude now
$$  \varepsilon_\rho(\widehat{t_s}\cdot \widehat{t_s'}) = X_\rho(t_s, t_s')^{\frac{d}{2}} = \zeta_{m_s}^{\frac{d}{2}} =-1$$
which yields the proof for $\rm{rk}_2(G/Z)=2.$

\end{proof}

\begin{rem} If $\rho=(Z,\chi_\rho)$ is a Heisenberg representation and if
$\omega:G/[G,G]\rightarrow\bbC^\times$ is a 1-dimensional character, then
$$  \rho\otimes\omega = (Z,\chi_{\rho\otimes\omega})\quad\textrm{where} \; 
\chi_{\rho\otimes\omega}=\chi_\rho\cdot\omega_Z,$$
where $\omega_Z:=\omega|_{Z}$.\\
Together with the definition of $X_\rho$ which has been given in Lemma \ref{Lemma 4.1}, this shows that:
$\rho'=\rho\otimes\omega$ if and only if $X_{\rho'}=X_\rho$.\\
Moreover if $\rho:G\rightarrow GL_d(\bbC)$ is any representation of dimension $d,$ then
$$  det(\rho\otimes\omega) = det(\rho)\cdot \omega^d.$$
\end{rem}


\begin{prop}
Let $\rho =(Z,\chi_\rho)$ be a Heisenberg representation of $G$ of dimension $d$.
Then the following are equivalent:\\
{\bf (i)}  $\varepsilon_\rho(g)\chi_\rho(g^d)=1$ for all $g\in G$ such that $g^d \in [G,G].$\\
{\bf (ii)} There exists a 1-dimensional character $\omega$ of $G$ such that $det(\rho\otimes\omega)\equiv 1.$
\end{prop}
\begin{proof}
We know already that $\delta(g):= \varepsilon_\rho(g)\chi_\rho(g^d)$ is a character of 
$A:=G/[G,G]$ because it is the determinant of $\rho.$ If now $\delta(g)=1$ in any case where $g^d\in [G,G],$ then
considered as a character of $A$ we have $\delta(a)=1$ always if $a^d =1.$ But $A$ is a finite abelian group thus
by duality this means $\delta = \eta^d$ for some character $\eta$ of $A$. Now put  $\omega= \eta^{-1}$ and consider
it as a character of $G.$ Then it follows that
$$   \det(\rho\otimes\omega)\equiv 1.$$
The implication $(ii) \Rightarrow: (i)$ is obvious because $\det(\rho\otimes\omega) = det(\rho)\cdot\omega^d.$
\end{proof}

\begin{exm}\label{Example 4.7}

 If $\rho=(Z,\chi_\rho)$ is a Heisenberg representation (of dimension $>1$) 
 for a nonabelian group of order $p^3, (p\neq 2,)$ then $Z=[G,G]$
 is cyclic group of order $p$, and $G^p=Z$ or $G^p=\{1\}$ depending on the isomorphism type of $G$. So by using 
 Theorem \ref{Theorem 1.1}, we can observe that  
 $\mathrm{det}(\rho)\not\equiv 1 $ and $\mathrm{det}(\rho)\equiv 1$ depending on the isomorphism type of $G$.



 We know that there are two nonabelian group of order $p^3$, up to isomorphism (for details see \cite{KC}). Now put 
\begin{equation*}
 G_{p}=\left\{\begin{pmatrix}
  1 & a & b  \\
  0 & 1 & c  \\
  0 & 0 & 1
        \end{pmatrix}: a,b,c\in \mathbb{Z}/p\mathbb{Z}\right\}
\end{equation*}
We observe that this $G_p$ is a nonabelian group under matrix multiplication of order $p^3$. We also see that 
$G_{p}^{p}=\{I_3\}$, the identity
in $G_{p}$. Now if $\rho$ is a Heisenberg representation of dimension $\ne 1$ of group $G_p$, 
then we will have $\mathrm{det}(\rho)\equiv 1$.

And when $G$ is an extraspecial group of order $p^3$, where $p\neq 2$ with $G^p=Z$, we will have 
 $\mathrm{Ker}(\Psi)\cong C_p\times C_p$, where $C_p$ is the cyclic group of order $p$.
 Therefore $\mathrm{det}(\rho)(g)=\chi_{Z}(g^p)$. This shows that 
$\mathrm{det}(\rho)\not\equiv 1$.

 Here, we observe that 
for nonabelian group of order $p^3$, where $p$ is prime, 
the determinant of Heisenberg representation of  $G$ gives the information
about the isomorphism type of $G$.
 
\end{exm}

In the following example we see if $G$ is nonabelian 
of order 8, then $det \;\rho\not\equiv 1$ if $G$ is dihedral and $det \;\rho\equiv 1$ if $G$ is quaternion.

\begin{exm}
If $G$ is a nonabelian group of order 8 then it has by obvious reasons precisely one faithful Heisenberg 
representation $\rho=(Z,\chi_\rho)$ which is of dimension 2. Moreover $G^2Z=Z=[G,G]$ is of order 2 and $G/G^2Z$ is
Klein's 4-group. Therefore
$$  (\det\; \rho)(g) = \chi_\rho(g^2)=1\;\textrm{if}\;g\in Z\qquad (\det\; \rho)(g) = 
-\chi_\rho(g^2)\;\textrm{if}\;g\notin Z.$$
Now if $G$ is dihedral we have elements $g\notin Z$ of order 2 (the reflections), hence\\ 
$(\det\;\rho)(g)=-1$, 
whereas $G$ quaternion implies that all $g\notin Z$ are of order 4, hence $(\det\;\rho)(g)=(-1)(-1)=1.$

\end{exm}


\section{{\bf An arithmetic example}}

In this section we consider Heisenberg representations from the arithmetic point of view
(as in \cite{Z2}, pp. 301-302), and we can add now a result on the determinant.

Now we let $F/\bbQ_p$ be a p-adic number field, i.e. a finite extension of the field of rational
p-adic numbers and $G=G_F:=\rm{Gal}(F_{sep}/F)$ the absolute Galois group over $F$. Since $G_F$ is
profinite, a
\begin{center}
 Heisenberg representation $\rho=(Z,\chi_\rho)$ of $G_F$
\end{center}
is actually a representation of a finite factor group $\overline{G_F}=G_F/\rm{Ker}(\rho)=\rm{Gal}(K_\rho/F)$.
Moreover we have: $Z=G_K$ for an abelian extension $K/F$ such that $\rm{Gal}(K_\rho/K)=Z/\rm{Ker}(\chi_\rho)$
is cyclic. Thus $\rho$ can be considered as representation of $G_F/[G_K,G_K]=\rm{Gal}(K_{ab}/F)$ where
$K_{ab}$ denotes the maximal abelian extension of $K$.

Now by the Theorem of Shafarevich-Weil, (see \cite{AT}, p. 246, Theorem 6), the Galois group $\rm{Gal}(K_{ab}/F)$
can be identified as the profinite closure of the relative Weil-group
\begin{equation}\label{eqn 5.1}
 1 \to K^\times \to W_{K/F} \to \rm{Gal}(K/F ) \to 1
\end{equation}
corresponding to the fundamental class $\alpha_{K/F}\in H^2(\rm{Gal}(K/F),K^\times)$. And the reciprocity
map $g\in G_F/[G_F,G_F]\to x_g\in\widetilde{F^\times}$(=profinite completion of $F^\times$) can be recovered as the
transfer map
\begin{equation}\label{eqn 5.2}
T_{K/F} : W_{K/F} /[W_{K/F} , W_{K/F}] \xrightarrow{\sim} {K^\times}^{\rm{Gal}(K/F)}=F^\times,
\end{equation}
where we use $T_{K/F}$ to denote the transfer from $W_{K/F}$ to the abelian subgroup $K^\times$ which
is actually an isomorphism between $W_{K/F}/[W_{K/F}, W_{K/F}]$ and $F^\times$(See \cite{AT}, p.239 where
the notation $T_{K/F}=V_{K/F}$ has been used). Moreover from (\ref{eqn 5.1}) we obtain
\begin{equation}\label{eqn 5.3}
 [W_{K/F} , W_{K/F}] = K_F^\times\subset K^\times,
\end{equation}
where $K_F^\times:=\{y\in K^\times| N_{K/F}(y)=1\}$ denotes the kernel of the norm map. (see
\cite{AT}, p.182, Theorem 3 with G abelian and $\alpha=\alpha_{K/F}$. We will use the notation $N_{K/F}$
also to denote the subgroup of norms: $N_{K/F}=N_{K/F}(K^\times)\subset F^\times$.)\\
Our point view is now to identify $\rho=(Z,\chi_\rho)$ with a representation
\begin{equation}\label{eqn 5.4}
\rho: W_{K/F}\to GL_d(\bbC),\qquad\text{ where $K$ is the fixed point field of $Z\subset G_F$}.
\end{equation}
Then via (\ref{eqn 5.2}) the determinant $\det(\rho)$ identifies with a character 
$\det(\rho):F^\times\to\bbC^\times$. And
the central character $\chi_\rho:Z\to\bbC^\times$ identifies with a character $\chi_K$ of 
$\widetilde{K^\times}\cong Z/[Z,Z]$.
But actually $\chi_\rho$ is a character of $Z/[Z,G_F]$ which translates into 
$\chi_K:K^\times/I_FK^\times\to\bbC^\times$, where
$I_FK^\times\subset K_F^\times$ is the subgroup generated by all $y^{\sigma-1}$, for $y\in K^\times$,
$\sigma\in\rm{Gal}(K/F)$. This is
because $y^{\sigma-1}=[y,\sigma]$ can be interpreted as a commutator in $W_{K/F}$.

Moreover we note that the subgroup $G_F^2Z\subseteq G_F$ corresponds to the field extension $K_2/F$
which comes as the composite of all quadratic subextensions $E/F$ inside $K$. And the corresponding 
$G_F^2Z/[Z,Z]$ identifies with the subgroup $W_{K/K_2}\subseteq W_{K/F}$. This leads us to the
following reformulation of Theorem \ref{Theorem 1.1}:
\begin{prop}\label{Proposition 5.1}
 Let $\rho=(Z,\chi_\rho)=(G_K,\chi_K)$ be a Heisenberg representation of dimension $d$ of the Galois group 
$G_F$ which we may assume to be given via (\ref{eqn 5.4}). Then $\det(\rho)$
understood as a character of $F^\times$ has the invariant form:
\begin{equation}\label{eqn 5.5}
 \det(\rho)(x)=\varepsilon_\rho(x)\cdot \chi_K(w_x^d),\quad \text{for all $x\in F^\times$},
\end{equation}
where $w_x\in W_{K/F}$ is any representative such that $T_{K/F}(w_x)=x$ and where
\begin{equation*}
 \varepsilon_\rho(x)=\begin{cases}
                      -1 & \text{if $[K_2:F]=4$ and $x\not\in N_{K_2/F}$},\\
                      1 & \text{in all other cases}.
                     \end{cases}
\end{equation*}
\end{prop}

\begin{proof}
We use the commutative diagram
\begin{equation*}
  \begin{CD}
  W_{K/F}   @>T_{K/F}>>  F^\times\\
  @VV V                      @VV V\\
  \rm{Gal}(K_{ab}/F) @>>> \rm{Gal}(F_{ab}/F)
 \end{CD}
 \end{equation*}
to shift our main result from the lower row to the upper row.
As to the definition of $\varepsilon_\rho$ we only need to remark that
$$G_F/G_F^2Z\cong W_{K/F}/W_{K/K_2}
\stackrel{\rm{T_{K/F}}}{\longleftrightarrow}
F^\times/N_{K_2/F},$$
where the second isomorphism is due to $N_{K_2/F}\circ T_{K/K_2}=T_{K/F}$ which can easily be seen
from: \cite{AT} , pp. 242-243, Theorem 4, (where the map $f$ is inverse to our isomorphism (\ref{eqn 5.2})).

Since $G_F/Z=\rm{Gal}(K/F)$ has the property $\sigma^{d}=1$ for all $\sigma\in\rm{Gal}(K/F)$, we deduce from
(\ref{eqn 5.1}) that $w_x^d\in K^{\times}$.
And from (\ref{eqn 5.2}), (\ref{eqn 5.3}) we see that any other $w_x'$ representing $x$ has the form 
$w_x'=w_xy$ for some $y\in K_F^\times$. Now modulo $I_FK^\times$ the two factors commute, hence
$${w_x'}^d\equiv w_x^dy^d \pmod{I_FK^\times},\quad \chi_K({w_x'}^d)=\chi_K(w_x^d)\chi_K(y^d).$$
So to recover the fact that $\chi_K(w_x^d)$ is well defined it will be enough to verify that $y^d\in I_FK^\times$,
if $y\in K_F^\times$. Here we use that $\overline{W_{K/F}}:=W_{K/F}/I_FK^\times$ is 2-step nilpotent because it
corresponds to $G_F/[Z,G_F]$ (see Remark \ref{Remark 2.2}(iii)). Therefore the commutator in
$\overline{W_{K/F}}$ induces an alternating bilinear form
\begin{equation}\label{eqn 5.6}
 c_{K/F}:F^\times\wedge F^\times\twoheadrightarrow K_F^\times/I_FK^\times,\quad
 c_{K/F}(x_1,x_2)=[w_{x_1},w_{x_2}],
\end{equation}
where now the representatives $w_{x_i} (i=1,2)$ and the commutator have to be considered in $\overline{W_{K/F}}$.
But the transfer map (\ref{eqn 5.2}) restricted to $K^\times\subset W_{K/F}$ is nothing else than the norm map:
$T_{K/F}(y)=N_{K/F}(y)$, and $\overline{K^\times}:=K^\times/I_FK^\times$ is precisely the center in 
$\overline{W_{K/F}}$. Therefore:
\begin{lem}
The alternating bilinear form $c_{K/F}$ has the radical $N_{K/F}\subset F^\times$ and induces
an isomorphism
\begin{equation}\label{eqn 5.7}
 \rm{Gal}(K/F)\wedge\rm{Gal}(K/F)\cong F^\times/N_{K/F}\wedge F^\times/N_{K/F}\xrightarrow{c_{K/F}}K_F^\times/I_FK^\times.
\end{equation}
\end{lem}
From what has been said so far it becomes obvious that $c_{K/F}$ is a surjective homomorphism.
But one may verify that (\ref{eqn 5.7}) is nothing else than a reformulation of the Tate
isomorphism 
$$\widehat{H}^{-3}(\rm{Gal}(K/F),\bbZ)\xrightarrow{\cup_{\alpha_{K/F}}}\widehat{H}^{-1}(\rm{Gal}(K/F),K^\times)$$
for abelian extensions $K/F$.

Thus we see that together with $\rm{Gal}(K/F)$ also $K_F^\times/I_FK^\times$ is of exponent $d$, hence
$x\in F^\times\mapsto\chi_K(w_x^d)$ is a well defined map, but, as we have seen, need not be a homomorphism.

\end{proof}
{\bf Remark 1:} If $x=N_{K/F}(y)$ is a norm then (as we have seen already) we may take
$w_x=y\in K^\times\subset W_{K/F}$. Moreover $N_{K/F}\subseteq N_{K_2/F}$, hence $\varepsilon_\rho(x)=1$,
and therefore the Proposition \ref{Proposition 5.1} implies:
$$\det(\rho)(N_{K/F}(y))=\chi_K^d(y)\quad\text{for all $y\in K^\times$}.$$
In group theoretic terms this means $\det(\rho)(z)=\chi_\rho^d(z)$ if $\rho=(Z,\chi_\rho)$ and 
$z\in Z$.

{\bf Remark 2:} If we actually want to compute $\chi_K(w_x^d)$ then it will be necessary to make
choices again. For this we consider
$$X:=\chi_K\circ c_{K/F}:F^\times/N_{K/F}\wedge F^\times/N_{K/F}\to \bbC^\times,$$
which again is non-degenerate and in algebraic terms corresponds to
$$\chi_\rho\circ [.,.]:G_F/Z\wedge G_F/Z\to\bbC^\times.$$
Now given $x\in F^\times$, we use Lemma \ref{Lemma 4.1} and find
$x\in N_{L/F}\subset F^\times$, for some subgroup of norms $N_{L/F}\subset F^\times$ which is maximal isotropic for
$X$. Let now $y_L\in L^\times$ such that
$$N_{L/F}(y_L)=x.$$
Using the commutative diagram
\begin{equation*}
  \begin{CD}
  W_{K/L}   @>\subset>>  W_{K/F}\\
  @VV T_{K/L}V                      @VV T_{K/F}V\\
  L^\times @>N_{L/F}>> F^\times
 \end{CD}
 \end{equation*}
with vertical transfer maps, we see that
\begin{equation}
 w_x:=T_{K/L}^{-1}(y_L)\in W_{K/L}\subset W_{K/F}\tag{*}
\end{equation}
is a representative for x because:
$$T_{K/F}(w_x)=N_{L/F}\circ T_{K/L}(w_x)=N_{L/F}(y_L)=x.$$
But $L^\times/N_{K/L}\cong\rm{Gal}(K/L)$ and $[K : L] = [L : F ] = d$ implies $y_L^d\in N_{K/L}$ hence
$$y_L^d=N_{K/L}(y_K)=T_{K/L}(y_K)$$ and
$$\chi_K(w_x^d)=_{(*)}\chi_K(T_{K/L}^{-1}(y_L^d))=\chi_K(N_{K/L}^{-1}(y_L^d))=\chi_K(N_{K/L}^{-1}((N_{L/F}^{-1}(x))^d).$$

{\bf Remark 3:} From \cite{AT}, it becomes obvious that Proposition \ref{Proposition 5.1}
generalizes to all situations where the base field $F$ is member of a class formation (cf. \cite{AT}
p.209), in particular if $F$ is a global field. Now we have to replace $F^\times$ by the formation
module $A_F$ which is assigned to $F$; for global fields $F$ it is $A_F=I_F/F^\times$ the idele class
group. Then we obtain the norm-residue-map of class field theory as a canonical map $A_F\to \rm{Gal}(F_{ab}/F)$,
and (\ref{eqn 5.1})-(\ref{eqn 5.3}) reformulate as
\begin{equation}
1\to A_K\to W_{K/F}\to \rm{Gal}(K/F)\to 1,\qquad\text{assigned to $\alpha_{K/F}\in H^2(\rm{Gal}(K/F),A_K)$,}\tag{{\bf 1A}} 
\end{equation}
\begin{equation}
 T_{K/F}: W_{K/F}/[W_{K/F},W_{K/F}]\xrightarrow{\sim} A_K^{\rm{Gal}(K/F)}=A_F, \tag{{\bf 2A}}
\end{equation}
\begin{equation}
 [W_{K/F},W_{K/F}]=A_{N_{K/F}}\subset A_K,\tag{{\bf 3A}}
\end{equation}
where $A_{N_{K/F}}=\{y\in A_K| N_{K/F}(y)=1\}$ is the kernel of the norm map $N_{K/F}:A_K\to A_F$.\\
Now we will use the notation $N_{K/F}$ also to denote the subgroup of norms $N_{K/F}:=N_{K/F}(A_K)\subset A_F$.

A Heisenberg representation $\rho=(Z,\chi_\rho)$ of the Galois group $G_F$ rewrites now as
$$\rho=(G_K,\chi_K),$$
where $K$ is the fix point field of $Z$ and $\chi_K:A_K/I_FA_K\to \bbC^\times$ is a Galois invariant character
of $A_K$ which comes as the lift of $\chi_\rho:Z/[Z,G_F]\to\bbC^\times$. Again we may interpret $\rho$ as a 
representation of $W_{K/F}$ or more precisely of the 2-step-nilpotent group 
$\overline{W_{K/F}}:=W_{K/F}/I_FA_K$
and via $A_F\to G_F/[G_F,G_F]$ we may interpret $\det(\rho)$ as a character of $A_F$. Then we obtain:\\
{\bf Proposition 5.1.A} {\it  Let $\rho=(Z,\chi_\rho)=(G_K,\chi_K)$ be a Heisenberg representation of dimension $d$
of $G_F$, interpreted as
$$\rho: W_{K/F}\to GL_d(\bbC).$$
The character $\det(\rho):A_F\to C^\times$ is then given in invariant form as
$$\det(\rho)(x)=\varepsilon_\rho(x)\cdot\chi_K(w_x^d),$$
where again $w_x\in W_{K/F}$ is a representative such that $T_{K/F}(w_x)=x\in A_F$ and where $\varepsilon_\rho(x)$
is defined as before but with the convention $N_{K_2/F}=N_{K_2/F}(A_{K_2})\subset A_F$.
}
 
\vspace{1cm}
\newpage
\textbf{Acknowledgements.} I would like to thank Prof E.-W. Zink for suggesting this problem and his constant 
valuable advices. I also thank to my adviser Prof. Rajat Tandon for his continuous encouragement. I
extend my gratitude to Prof. Elmar Grosse-Kl\"{o}nne for providing very good mathematical environment during stay in
Berlin. I am also grateful to Berlin Mathematical School for their financial support.

\vspace{1cm}


\end{document}